\documentclass[12pt]{article}
\usepackage{amssymb}
\usepackage{amsmath,amsthm}
\usepackage{cite}
\usepackage[dvips]{graphicx}
\usepackage{hyperref}
\usepackage{color}
\usepackage{setspace}
\usepackage{enumerate}
\usepackage{tkz-graph}
\usepackage{tikz}
\hypersetup{colorlinks=true}

\hypersetup{colorlinks=true, linkcolor=blue, citecolor=blue,urlcolor=blue}

\newtheorem{remark}{Remark}

 \newtheorem{lemma}[remark]{Lemma}
 \newtheorem{theorem}[remark]{Theorem}
 
 \newtheorem{corollary}[remark]{Corollary}

\title{On the General Randi\'c index of polymeric networks \\ modelled by generalized Sierpi\'{n}ski graphs}

\author{Alejandro Estrada-Moreno and Juan A. Rodr\'{\i}guez-Vel\'{a}zquez 
\\
{\small Departament d'Enginyeria Inform\`atica i Matem\`atiques,}\\
{\small Universitat Rovira i Virgili,}  {\small Av. Pa\"{\i}sos
Catalans 26, 43007 Tarragona, Spain.} \\{\small
alejandro.estrada\@@urv.cat}, juanalberto.rodriguez\@@urv.cat 
}

\begin{document}
\maketitle

\begin{abstract}
The General Randi\'{c} index $R_\alpha$ of a simple graph
$G$ is defined as
\[
R_\alpha(G)=\sum_{v_{i}\sim v_{j}} \left(d(v_i)d(v_j)\right)^\alpha,
\]
where $d(v_i)$ denotes the degree of the vertex $v_i$.   Rodr\'{\i}guez-Vel\'{a}zquez and Tom\'as-Andreu [MATCH Commun. Math. Comput. Chem. 74 (1) (2015) 145--160] obtained closed formulae for the Randi\'{c} index $R_{-1/2}$ of  Sierpi\'nski-type polymeric networks, where the base graph is a complete graph, a triangle-free regular graph or a bipartite semiregular graph. 
In the present article  we obtain closed  formulae for the general Randi\'{c} index $R_{\alpha}$ of Sierpi\'nski-type polymeric networks, where the base graph is arbitrary.
\end{abstract}

\section{Introduction}

Around the middle of the last century theoretical chemists proposed the use of topological indices to obtain information on the dependence of various properties of organic substances on molecular structure.  In this sense, a large number of various topological indices was proposed and considered in the chemical literature \cite{Todeschini2008}. We highlight the article
\cite{camarda}  where Camarda and Maranas  addressed the design of polymers with optimal levels of macroscopic properties through
the use of topological indices. Specifically, in the above mentioned article two zeroth-order and two first-order connectivity indices
were  employed for the first time as descriptors in structure-property correlations in an
optimization study. Based on these descriptors, a set of new correlations for heat capacity,
cohesive energy, glass transition temperature, refractive index, and dielectric constant were proposed.

The molecular structure-descriptor, introduced in 1975 by Milan Randi\'c in  \cite{Randic1975}, is defined as 
$$R(G)=\sum_{v_iv_j\in E}\frac{1}{\sqrt{d(v_i)d(v_j)}},$$ where $d(v_i)$ represents the degree of the vertex $v_i$ in $G=(V,E)$. Nowadays, $R(G)$
is referred to as the \textit{Randi\'c index} of $G$.
This graph topological index, sometimes referred to as \emph{connectivity index}, has been successfully related to a variety of physical, chemical, and pharmacological properties of organic molecules and became one of the most popular molecular-structure descriptors \cite{Randic2008}. After the publication of the first paper \cite{Randic1975}, mathematical properties and generalizations of $R(G)$ were extensively studied, for instance, see \cite{gao-lu,Gutman2008,hu-li,Li2005,Li2006,Li2008,liu-gutman, J.Liu,Rodriguez2005a,Rodriguez2005b,Rodriguez-Velazquez2015,Yero2010b} and the references cited therein. The Randi\'c index was generalized by Gutman and Lepovi\'c in \cite{Gutman2001} as $$R_\alpha(G)=\sum_{v_iv_j\in E}\left(d(v_i)d(v_j)\right)^\alpha,\; \alpha\ne 0.$$ Obviously, the standard Randi\'c index is obtained when $\alpha=-1/2$. In the chemical literature the quantity
$${\emph{R}}_1(G)=\sum_{v_iv_j\in E}d(v_i)d(v_j).$$
is called  {\em the second Zagreb index} \cite{Das2004}. 
 
Some topological indices have been studied also for the case of polymeric networks. For instance,  we cite the article \cite{Wang}, where the authors gave the explicitly formula of the $k$-connectivity index of an infinite class of dendrimer  nanostars. 

Over the past three decades, polymer networks has emerged as a coherent subject area \cite{Blumen,JU1,JU,Rodriguez-Velazquez2015,LibroPolymers}.
While the basic works on polymer modelling
started from linear polymeric systems, in recent years the attention has focused more
and more on complex underlying geometries including fractal-type networks. 
 It is well-known that, in comparison with those linear polymers, the properties of polymer networks depend to a much larger extent on methods and condition of preparation, \textit{i.e.}, properties depend not only on the chemical structure of the individual polymer chains, but on how those chains are joined together to form a network \cite{LibroPolymers}.
In this article we consider a
model of polymer networks based on generalized Sierpi\'{n}ski graphs, which was previously studied in \cite{Rodriguez-Velazquez2015}.

To begin with, we need some notation and terminology. Let $G=(V,E)$ be a non-empty graph of order $n$ and vertex set $V=\{1,2,\ldots,n\}$. We denote by $\{1,2,\ldots,n\}^t$ the set of words of length $t$ on alphabet $\{1,2,\ldots,n\}$. The letters of a word $u$ of length $t$ are denoted by $u_1u_2\cdots u_t$. The concatenation of two words $u$ and $v$  is denoted by $uv$. Klav\v{z}ar and Milutinovi\'c introduced in \cite{Klavzar1997} the graph  $S(K_n, t)$ whose vertex set is $\{1,2,\ldots,n\}^t$, where  
$\{u,v\}$ is an edge if and only if there exists $i\in \{1,\ldots,t\}$ such that:
\begin{center}
(i)  $u_j=v_j$, if  $j<i$; (ii)  $u_i\ne v_i$; (iii) $u_j=v_i$ and $v_j=u_i$ if $j>i$.
\end{center}

The graph $S(K_3,t)$ is isomorphic to 
the graph of the  Tower of Hanoi with $t$ disks \cite{Klavzar1997}. Later, those graphs have been
called Sierpi\'{n}ski graphs in \cite{Klavzar2002} and they were studied by now from numerous points of view. The reader is
invited to read, for instance, the following recent papers \cite{Gravier...Parreau,Hinz-Parisse,Hirtz-Holz,Klavzar2002,Klavsar-Peterin,Klavsar-Zeljic,Parisse2009,Klavzar2016(survey)} and references
therein.
 
This construction was generalized in \cite{GeneralizedSierpinski} for any graph $G=(V,E)$, by defining the \emph{generalized Sierpi\'{n}ski graph},  $S(G,t)$,  as the graph with vertex set $\{1,2,\ldots,n\}^t$ and edge set defined as follows. $\{u,v\}$ is an edge if and only if there exists $i\in \{1,\ldots,t\}$ such that:
 \begin{center}
 (i) $u_j=v_j$, if  $j<i$; (ii) $u_i\ne v_i$ and $\{u_i,v_i\}\in E$; (iii) $u_j=v_i$ and $v_j=u_i$ if $j>i$.
 \end{center}

Notice that if $\{u,v\}$ is an edge of $S(G,t)$, there is an edge $\{x,y\}$ of $G$ and a word $w$ such that $u=wxyy\cdots y$ and $v=wyxx\cdots x$. In general, $S(G,t)$ can be constructed recursively from $G$ with the following process: $S(G,1)=G$ and, for $t\ge 2$, we copy $n$ times $S(G, t-1)$ and add the letter $x$ at the beginning of each label of the vertices belonging to  the copy of $S(G,t-1)$ corresponding to $x$. Then for every edge $\{x,y\}$ of $G$, add an edge between vertex $xyy\cdots y$ and vertex $yxx\cdots x$. See, for instance, Figure \ref{FigS(G,2)} that shows a graph $G$ and the generalized Sierpi\'{n}ski graph $S(G,2)$. Besides Figure \ref{FigS(G,3).} shows the generalized Sierpi\'{n}ski graph $S(G,3)$.  Vertices of the form $xx\cdots x$ are called \textit{extreme vertices}. Notice that for any graph $G$ of order $n$ and any integer $t\ge 2$,  $S(G,t)$  has $n$ extreme vertices and, if $x$ has degree $d(x)$ in $G$, then the extreme vertex $xx\cdots x$ of $S(G,t)$   also has degree  $d(x)$. Moreover,   the degrees of two vertices $yxx\cdots x$ and  $xyy\cdots y$, which connect two copies of $S(G,t-1)$, are  equal to  $d(x)+1$ and $d(y)+1$, respectively.

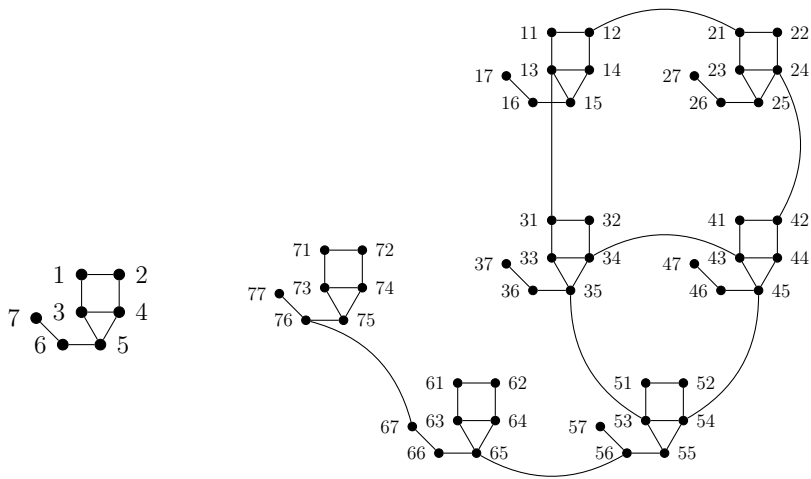
\begin{figure}[ht]
\centering
\begin{tikzpicture}[transform shape, inner sep = .5mm]
%%%%%%%Base graph G%%%%%%%%%%%%%%%%%%%%%%%%%%%%%%%%%%%%%%
\def\side{.5};
\pgfmathsetmacro\radius{\side/sqrt(3)};
\foreach \ind in {3,4,5}
{
\pgfmathparse{150-120*(\ind-3)};
\node [draw=black, shape=circle, fill=black] (\ind) at (\pgfmathresult:\radius cm) {};
}
\node [draw=black, shape=circle, fill=black] (1) at ([yshift=\side cm]3) {};
\node [draw=black, shape=circle, fill=black] (2) at ([yshift=\side cm]4) {};
\node [draw=black, shape=circle, fill=black] (6) at ([xshift=-\side cm]5) {};
\node [draw=black, shape=circle, fill=black] (7) at ([shift=({135:\side cm})]6) {};
\foreach \ind in {2,4,5}
{
\node [scale=.8] at ([xshift=.3 cm]\ind) {$\ind$};
}
\foreach \ind in {1,3,6,7}
{
\node [scale=.8] at ([xshift=-.3 cm]\ind) {$\ind$};
}
\foreach \u/\v in {1/2,1/3,2/4,3/4,3/5,4/5,5/6,6/7}
{
\draw (\u) -- (\v);
}

%%%%%%%%%%%S(G,2)%%%%%%%%%%%%%%%%%%%%%%%%%%%%%%%%%%%%%%%%%%

\def\widenodetwo{.4};
\pgfmathparse{15*\side};
\node (center) at (\pgfmathresult cm,0) {};
\pgfmathsetmacro\sideTwo{5*\side};
\pgfmathsetmacro\radiuscenter{\sideTwo/sqrt(3)};
\foreach \d in {3,4,5}
{
\pgfmathparse{150-120*(\d-3)};
\node (center\d) at ([shift=({\pgfmathresult:\radiuscenter cm})]center) {};
\foreach \ind in {3,4,5}
{
\pgfmathparse{150-120*(\ind-3)};
\node [draw=black, shape=circle, fill=black,inner sep = \widenodetwo mm] (\d\ind) at ([shift=({\pgfmathresult:\radius cm})]center\d) {};
}
\node [draw=black, shape=circle, fill=black,inner sep = \widenodetwo mm] (\d1) at ([yshift=\side cm]\d3) {};
\node [draw=black, shape=circle, fill=black,inner sep = \widenodetwo mm] (\d2) at ([yshift=\side cm]\d4) {};
\node [draw=black, shape=circle, fill=black,inner sep = \widenodetwo mm] (\d6) at ([xshift=-\side cm]\d5) {};
\node [draw=black, shape=circle, fill=black,inner sep = \widenodetwo mm] (\d7) at ([shift=({135:\side cm})]\d6) {};
\foreach \ind in {2,4,5}
{
\node [scale=.6] at ([xshift=.3 cm]\d\ind) {$\d\ind$};
}
\foreach \ind in {1,3,6,7}
{
\node [scale=.6] at ([xshift=-.3 cm]\d\ind) {$\d\ind$};
}
\foreach \u/\v in {1/2,1/3,2/4,3/4,3/5,4/5,5/6,6/7}
{
\draw (\d\u) -- (\d\v);
}
}
%%%%%1x
\foreach \d in {1,2,6,7}
{
\ifthenelse{\d=1}
{
\node (center1) at ([shift=({0,\sideTwo})]center3) {};
}
{
\ifthenelse{\d=2}
{
\node (center2) at ([shift=({0,\sideTwo})]center4) {};
}
{
\ifthenelse{\d=6}
{
\node (center6) at ([shift=({-\sideTwo,0})]center5) {};
}
{
\node (center7) at ([shift=({135:\sideTwo cm})]center6) {};
};
};
};
\foreach \ind in {3,4,5}
{
\pgfmathparse{150-120*(\ind-3)};
\node [draw=black, shape=circle, fill=black,inner sep = \widenodetwo mm] (\d\ind) at ([shift=({\pgfmathresult:\radius})]center\d) {};
}

\node [draw=black, shape=circle, fill=black,inner sep = \widenodetwo mm] (\d1) at ([yshift=\side cm]\d3) {};
\node [draw=black, shape=circle, fill=black,inner sep = \widenodetwo mm] (\d2) at ([yshift=\side cm]\d4) {};
\node [draw=black, shape=circle, fill=black,inner sep = \widenodetwo mm] (\d6) at ([xshift=-\side cm]\d5) {};
\node [draw=black, shape=circle, fill=black,inner sep = \widenodetwo mm] (\d7) at ([shift=({135:\side cm})]\d6) {};
\foreach \ind in {2,4,5}
{
\node [scale=.6] at ([xshift=.3 cm]\d\ind) {$\d\ind$};
}
\foreach \ind in {1,3,6,7}
{
\node [scale=.6] at ([xshift=-.3 cm]\d\ind) {$\d\ind$};
}
\foreach \u/\v in {1/2,1/3,2/4,3/4,3/5,4/5,5/6,6/7}
{
\draw (\d\u) -- (\d\v);
}
}
\foreach \u/\v in {1/2,1/3,2/4,3/4,3/5,4/5,5/6,6/7}
{
\ifthenelse{\u=1 \AND \v>2}
{
\draw (\u\v) -- (\v\u);
}
{
\ifthenelse{\u=3 \AND \v=5 \OR \u=6}
{
\draw (\u\v) to[bend right] (\v\u);
}
{
\draw (\u\v) to[bend left] (\v\u);
};
};
}
\end{tikzpicture}
\caption{A graph $G$ and the Sierpi\'{n}ski graph $S(G,2)$.}
\label{FigS(G,2)}
\end{figure}

%%%%%%%%%%%%%%%%%%%%%%%%%%%%%%%%%%%%%%%%%%%%%%%%%%%%%%%%%%%
%%%%%%%Aqui esta la division%%%%%%%%%%%%%%%%%%%%%%%%%%%%%%%
%%%%%%%%%%%%%%%%S(G,3)%%%%%%%%%%%%%%%%%%%%%%%%%%%%%%%%%%%%%

\begin{figure}[ht]
\begin{tikzpicture}[transform shape, inner sep = .3mm]

%%%%%%%%%%%S(G,3)%%%%%%%%%%%%%%%%%%%%%

\def\side{.3};
\pgfmathsetmacro\radius{\side/sqrt(3)};
\pgfmathsetmacro\sideTwo{5*\side};
\pgfmathsetmacro\radiuscenter{\sideTwo/sqrt(3)};
\pgfmathsetmacro\sideThree{16*\side};
\pgfmathsetmacro\radiuscenterThree{\sideThree/sqrt(3)};
\node (center) at (0,0) {};
\foreach \c in {3,4,5}
{
\pgfmathparse{150-120*(\c-3)};
\node (center\c) at ([shift=({\pgfmathresult:\radiuscenterThree cm})]center) {};

%%%%%%%%Copies of S(G,2)%%%%%%%%%%

\foreach \d in {3,4,5}
{
\pgfmathparse{150-120*(\d-3)};
\node (center\c\d) at ([shift=({\pgfmathresult:\radiuscenter cm})]center\c) {};
\foreach \ind in {3,4,5}
{
\pgfmathparse{150-120*(\ind-3)};
\node [draw=black, shape=circle, fill=black] (\c\d\ind) at ([shift=({\pgfmathresult:\radius cm})]center\c\d) {};
}
\node [draw=black, shape=circle, fill=black] (\c\d1) at ([yshift=\side cm]\c\d3) {};
\node [draw=black, shape=circle, fill=black] (\c\d2) at ([yshift=\side cm]\c\d4) {};
\node [draw=black, shape=circle, fill=black] (\c\d6) at ([xshift=-\side cm]\c\d5) {};
\node [draw=black, shape=circle, fill=black] (\c\d7) at ([shift=({135:\side cm})]\c\d6) {};
\foreach \ind in {2,4,5}
{
\node [scale=.4] at ([xshift=.18 cm]\c\d\ind) {$\c\d\ind$};
}
\foreach \ind in {1,3,6,7}
{
\node [scale=.4] at ([xshift=-.18 cm]\c\d\ind) {$\c\d\ind$};
}
\foreach \u/\v in {1/2,1/3,2/4,3/4,3/5,4/5,5/6,6/7}
{
\draw (\c\d\u) -- (\c\d\v);
}
}
%%%%%11x
\foreach \d in {1,2,6,7}
{
\ifthenelse{\d=1}
{
\node (center\c1) at ([shift=({0,\sideTwo})]center\c3) {};
}
{
\ifthenelse{\d=2}
{
\node (center\c2) at ([shift=({0,\sideTwo})]center\c4) {};
}
{
\ifthenelse{\d=6}
{
\node (center\c6) at ([shift=({-\sideTwo,0})]center\c5) {};
}
{
\node (center\c7) at ([shift=({135:\sideTwo cm})]center\c6) {};
};
};
};
\foreach \ind in {3,4,5}
{
\pgfmathparse{150-120*(\ind-3)};
\node [draw=black, shape=circle, fill=black] (\c\d\ind) at ([shift=({\pgfmathresult:\radius})]center\c\d) {};
}

\node [draw=black, shape=circle, fill=black] (\c\d1) at ([yshift=\side cm]\c\d3) {};
\node [draw=black, shape=circle, fill=black] (\c\d2) at ([yshift=\side cm]\c\d4) {};
\node [draw=black, shape=circle, fill=black] (\c\d6) at ([xshift=-\side cm]\c\d5) {};
\node [draw=black, shape=circle, fill=black] (\c\d7) at ([shift=({135:\side cm})]\c\d6) {};
\foreach \ind in {2,4,5}
{
\node [scale=.4] at ([xshift=.18 cm]\c\d\ind) {$\c\d\ind$};
}
\foreach \ind in {1,3,6,7}
{
\node [scale=.4] at ([xshift=-.18 cm]\c\d\ind) {$\c\d\ind$};
}
\foreach \u/\v in {1/2,1/3,2/4,3/4,3/5,4/5,5/6,6/7}
{
\draw (\c\d\u) -- (\c\d\v);
}
}
\foreach \u/\v in {1/2,1/3,2/4,3/4,3/5,4/5,5/6,6/7}
{
\ifthenelse{\u=1 \AND \v>2}
{
\draw (\c\u\v) -- (\c\v\u);
}
{
\ifthenelse{\u=3 \AND \v=5 \OR \u=6}
{
\draw (\c\u\v) to[bend right] (\c\v\u);
}
{
\draw (\c\u\v) to[bend left] (\c\v\u);
};
};
}
}
\foreach \c in {1,2,6,7}
{
\ifthenelse{\c=1}
{
\node (center1) at ([shift=({0,\sideThree})]center3) {};
}
{
\ifthenelse{\c=2}
{
\node (center2) at ([shift=({0,\sideThree})]center4) {};
}
{
\ifthenelse{\c=6}
{
\node (center6) at ([shift=({-\sideThree,0})]center5) {};
}
{
\node (center7) at ([shift=({135:\sideThree cm})]center6) {};
};
};
};

%%%%%%%%S(G,2)%%%%%%%%%%

\foreach \d in {3,4,5}
{
\pgfmathparse{150-120*(\d-3)};
\node (center\c\d) at ([shift=({\pgfmathresult:\radiuscenter cm})]center\c) {};
\foreach \ind in {3,4,5}
{
\pgfmathparse{150-120*(\ind-3)};
\node [draw=black, shape=circle, fill=black] (\c\d\ind) at ([shift=({\pgfmathresult:\radius cm})]center\c\d) {};
}
\node [draw=black, shape=circle, fill=black] (\c\d1) at ([yshift=\side cm]\c\d3) {};
\node [draw=black, shape=circle, fill=black] (\c\d2) at ([yshift=\side cm]\c\d4) {};
\node [draw=black, shape=circle, fill=black] (\c\d6) at ([xshift=-\side cm]\c\d5) {};
\node [draw=black, shape=circle, fill=black] (\c\d7) at ([shift=({135:\side cm})]\c\d6) {};
\foreach \ind in {2,4,5}
{
\node [scale=.4] at ([xshift=.18 cm]\c\d\ind) {$\c\d\ind$};
}
\foreach \ind in {1,3,6,7}
{
\node [scale=.4] at ([xshift=-.18 cm]\c\d\ind) {$\c\d\ind$};
}
\foreach \u/\v in {1/2,1/3,2/4,3/4,3/5,4/5,5/6,6/7}
{
\draw (\c\d\u) -- (\c\d\v);
}
}
%%%%%11x
\foreach \d in {1,2,6,7}
{
\ifthenelse{\d=1}
{
\node (center\c1) at ([shift=({0,\sideTwo})]center\c3) {};
}
{
\ifthenelse{\d=2}
{
\node (center\c2) at ([shift=({0,\sideTwo})]center\c4) {};
}
{
\ifthenelse{\d=6}
{
\node (center\c6) at ([shift=({-\sideTwo,0})]center\c5) {};
}
{
\node (center\c7) at ([shift=({135:\sideTwo cm})]center\c6) {};
};
};
};
\foreach \ind in {3,4,5}
{
\pgfmathparse{150-120*(\ind-3)};
\node [draw=black, shape=circle, fill=black] (\c\d\ind) at ([shift=({\pgfmathresult:\radius})]center\c\d) {};
}

\node [draw=black, shape=circle, fill=black] (\c\d1) at ([yshift=\side cm]\c\d3) {};
\node [draw=black, shape=circle, fill=black] (\c\d2) at ([yshift=\side cm]\c\d4) {};
\node [draw=black, shape=circle, fill=black] (\c\d6) at ([xshift=-\side cm]\c\d5) {};
\node [draw=black, shape=circle, fill=black] (\c\d7) at ([shift=({135:\side cm})]\c\d6) {};
\foreach \ind in {2,4,5}
{
\node [scale=.4] at ([xshift=.18 cm]\c\d\ind) {$\c\d\ind$};
}
\foreach \ind in {1,3,6,7}
{
\node [scale=.4] at ([xshift=-.18 cm]\c\d\ind) {$\c\d\ind$};
}
\foreach \u/\v in {1/2,1/3,2/4,3/4,3/5,4/5,5/6,6/7}
{
\draw (\c\d\u) -- (\c\d\v);
}
}
\foreach \u/\v in {1/2,1/3,2/4,3/4,3/5,4/5,5/6,6/7}
{
\ifthenelse{\u=1 \AND \v>2}
{
\draw (\c\u\v) -- (\c\v\u);
}
{
\ifthenelse{\u=3 \AND \v=5 \OR \u=6}
{
\draw (\c\u\v) to[bend right] (\c\v\u);
}
{
\draw (\c\u\v) to[bend left] (\c\v\u);
};
};
}
}
\foreach \u/\v in {1/2,1/3,2/4,3/4,3/5,4/5,5/6,6/7}
{
\ifthenelse{\u=1 \AND \v>2}
{
\draw (\u\v\v) -- (\v\u\u);
}
{
\ifthenelse{\u=3 \AND \v=5 \OR \u=6}
{
\draw (\u\v\v) to[bend right] (\v\u\u);
}
{
\draw (\u\v\v) to[bend left] (\v\u\u);
};
};
}
\end{tikzpicture}
\caption{The Sierpi\'{n}ski graph $S(G,3)$ for the graph $G$ of Figure \ref{FigS(G,2)}. }
\label{FigS(G,3).}
\end{figure}
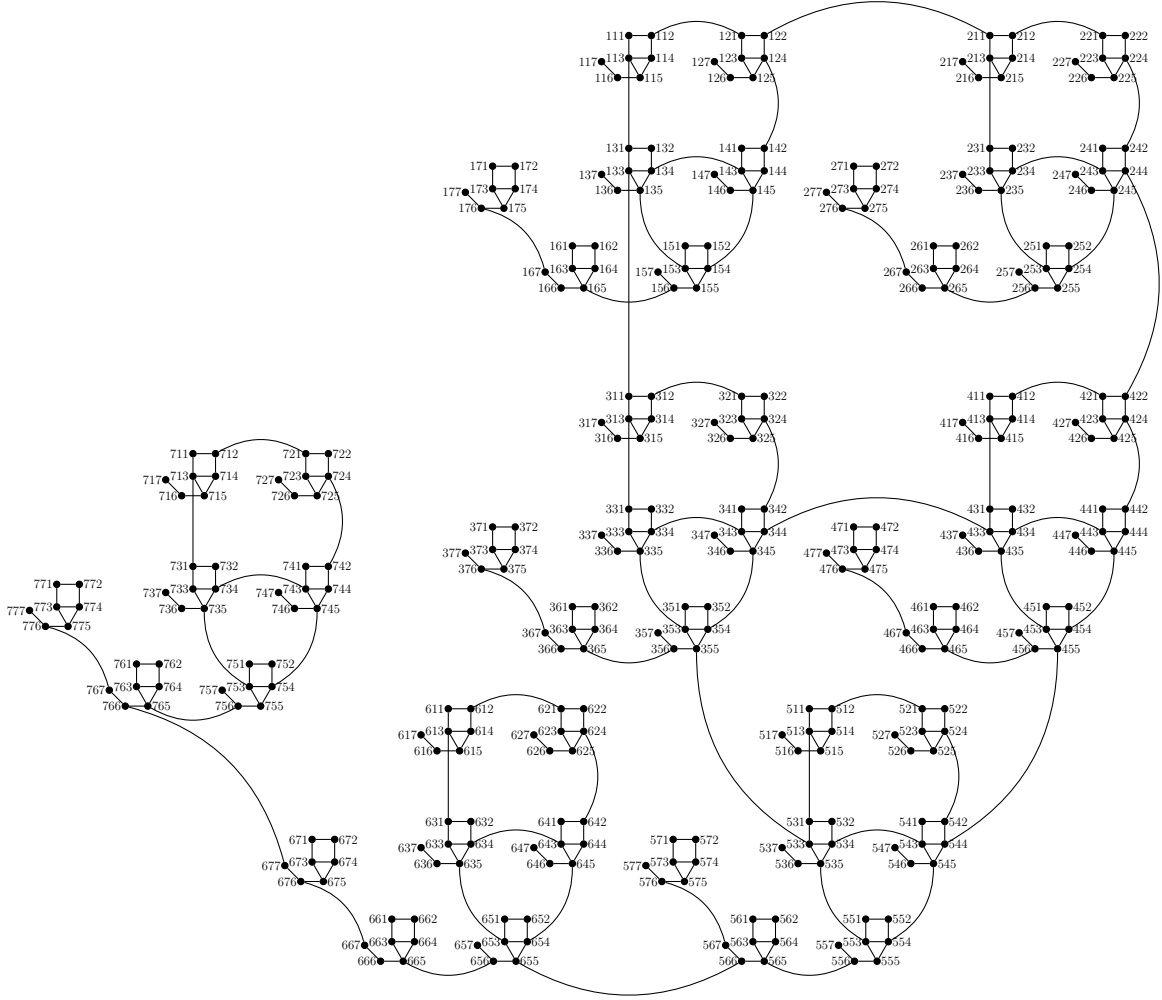

We denote by $P_r$ the path graph of order $r$.
Notice that for $G=K_2$  we obtain $S(K_2,2)=P_4$ and, in general, $S(K_2,t)=P_{2^t}$, which is the simplest possible polymer model presented by the ideal chain.
Also, the graphs  $S(K_n,t)$ were used in  \cite{Blumen,JU1,JU} to
analyse the scaling behaviour of experimentally accessible dynamical relaxation forms for polymers  modelled through finite Sierpi\'{n}ski-type graphs, which we denote here by $P(K_n,t)$. Using the approach developed in \cite{JU1,JU} to construct $P(K_n,t)$,  now we define the 
polymeric Sierpi\'{n}ski
graphs
 $P(G,t)=(V^*,E^*)$ introduced in \cite{Rodriguez-Velazquez2015}, where $G$ is a connected graph of order $n$ and $t$ is a positive integer. For $i\in \{1,\ldots,t\}$ we define the sets $A_i=\{a_{i_1}, \ldots, a_{i_{n^{i-1}}}\}$ and we denote $S(G,i)=(V_i,E_i)$ and $V_i=\{v_{i_1}, \ldots, v_{i_{n^{i}}}\}$. Then, the vertex set of $P(G,t)$ is $$V^*=\bigcup_{i=1}^t(A_i\cup V_i)$$ and the edge set of $P(G,t)$ is $$E^*=\left(\bigcup_{i=1}^n(E_i\cup B_i)\right)\cup \left(\bigcup_{i=1}^{t-1}C_i\right),$$ where $C_i=\{\{v_{i_j},a_{{(i+1)}_j}\}: j=1,\ldots,n^i\}$, $B_i=\bigcup_{j=1}^{n^{i-1}}W_j$, and $W_j$ is formed by the edges obtained by connecting $a_{i_j}$ to every vertex belonging to the $j$-th copy of $G$ in $S(G,i)$. In other words, we construct $P(G,t)$ as follows:  The iterative construction starts from one vertex, $a_{1_1}$, and one copy of $G=S(G,1)$. So, we obtain $P(G,1)$ by connecting $a_{1_1}$ to every vertex of $S(G,1)$. To obtain $P(G,2)$ we take $P(G,1)$, $A_2$ and $S(G,2)$. Then we connect each element $a_{2_j}\in A_2$ to $v_{1_j}\in V_1$ and  we also connect $a_{2_j}$ to every vertex in the $j$-th copy of $G$ in $S(G,2)$. Analogously, for the construction of  $P(G,t)$ we take $P(G,t-1)$, $A_t$ and $S(G,t)$. Then, we connect each element $a_{t_j}\in A_t$ to $v_{{t-1}_j}\in V_{t-1}$ and  we also connect $a_{t_j}$ to every vertex in the $j$-th copy of $G$ in $S(G,t)$. Notice that  $P(K_3,2)=S(K_{4},2)$, $S(K_3,2)=P(K_{2},2),$ while for $t\geq3, P(K_{n},t)\neq S(K_{n+1},t)$. Figure \ref{FigP(G,2)} shows a sketch of a polymeric Sierpi\'{n}ski graph $P(G,2)$.

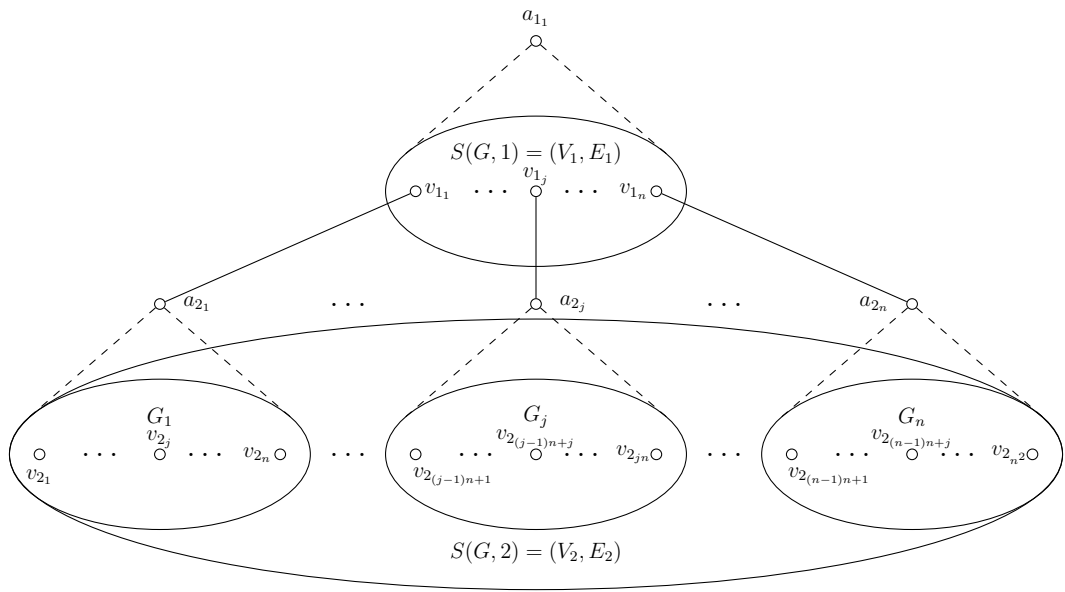
\begin{figure}[!ht]
\centering
\begin{tikzpicture}[transform shape, inner sep = .5mm]
\node [draw=black, shape=circle, fill=white] (a11) at (0,0) {};
\node [scale=.7] at ([yshift=.3 cm]a11) {$a_{1_1}$};
\draw (0,-2) ellipse (2cm and 1cm);
\pgfmathparse{sqrt(3)};
\draw[dashed, black] (a11) -- (-\pgfmathresult,-1.5);
\draw[dashed, black] (a11) -- (\pgfmathresult,-1.5);
\node [draw=black, shape=circle, fill=white] (v11) at (-1.6,-2) {};
\node [scale=.7] at ([xshift=.3 cm]v11) {$v_{1_1}$};
\node [draw=black, shape=circle, fill=white] (a21) at (-5,-3.5) {};
\node [scale=.7] at ([xshift=.5 cm]a21) {$a_{2_1}$};
\draw[black] (v11) -- (a21);
\node [draw=black, shape=circle, fill=white] (v1n) at (1.6,-2) {};
\node [scale=.7] at ([xshift=-.3 cm]v1n) {$v_{1_n}$};
\node [draw=black, shape=circle, fill=white] (a2n) at (5,-3.5) {};
\node [scale=.7] at ([xshift=-.5 cm]a2n) {$a_{2_n}$};
\draw[black] (v1n) -- (a2n);
\node at (-.6,-2) {$\ldots$};
\node at (.6,-2) {$\ldots$};
\node [draw=black, shape=circle, fill=white] (v1j) at (0,-2) {};
\node [scale=.7] at ([yshift=.2 cm]v1j) {$v_{1_j}$};
\node [scale=.7] at (0,-1.5) {$S(G,1)=(V_1,E_1)$};
\node [draw=black, shape=circle, fill=white] (a2j) at (0,-3.5) {};
\node [scale=.7] at ([xshift=.5 cm]a2j) {$a_{2_j}$};
\draw[black] (v1j) -- (a2j);
\node at (-2.5,-3.5) {$\ldots$};
\node at (2.5,-3.5) {$\ldots$};

\draw (0,-5.5) ellipse (2cm and 1cm);
\draw[dashed, black] (a2j) -- (-\pgfmathresult,-5);
\draw[dashed, black] (a2j) -- (\pgfmathresult,-5);
\node [draw=black, shape=circle, fill=white] (v2j1) at (-1.6,-5.5) {};
\node [scale=.7] at ([shift=({.5 cm,-.3})]v2j1) {$v_{2_{(j-1)n+1}}$};
\node [draw=black, shape=circle, fill=white] (v2jj) at (0,-5.5) {};
\node [scale=.7] at ([yshift=.2 cm]v2jj) {$v_{2_{(j-1)n+j}}$};
\node [draw=black, shape=circle, fill=white] (v2jn) at (1.6,-5.5) {};
\node [scale=.7] at ([xshift=-.3 cm]v2jn) {$v_{2_{jn}}$};
\node at (-.8,-5.5) {$\ldots$};
\node at (.6,-5.5) {$\ldots$};
\node [scale=.7] at (0,-5) {$G_j$};

\draw (5,-5.5) ellipse (2cm and 1cm);
\pgfmathparse{sqrt(3)+5};
\draw[dashed, black] (a2n) -- (\pgfmathresult,-5);
\pgfmathparse{-sqrt(3)+5};
\draw[dashed, black] (a2n) -- (\pgfmathresult,-5);
\node [draw=black, shape=circle, fill=white] (v2n1) at (3.4,-5.5) {};
\node [scale=.7] at ([shift=({.5 cm,-.3})]v2n1) {$v_{2_{(n-1)n+1}}$};
\node [draw=black, shape=circle, fill=white] (v2nj) at (5,-5.5) {};
\node [scale=.7] at ([yshift=.2 cm]v2nj) {$v_{2_{(n-1)n+j}}$};
\node [draw=black, shape=circle, fill=white] (v2nn) at (6.6,-5.5) {};
\node [scale=.7] at ([xshift=-.3 cm]v2nn) {$v_{2_{n^2}}$};
\node at (4.2,-5.5) {$\ldots$};
\node at (5.6,-5.5) {$\ldots$};
\node [scale=.7] at (5,-5) {$G_n$};

\draw (-5,-5.5) ellipse (2cm and 1cm);
\pgfmathparse{sqrt(3)-5};
\draw[dashed, black] (a21) -- (\pgfmathresult,-5);
\pgfmathparse{-sqrt(3)-5};
\draw[dashed, black] (a21) -- (\pgfmathresult,-5);
\node [draw=black, shape=circle, fill=white] (v21) at (-6.6,-5.5) {};
\node [scale=.7] at ([yshift=-.3 cm]v21) {$v_{2_1}$};
\node [draw=black, shape=circle, fill=white] (v2j) at (-5,-5.5) {};
\node [scale=.7] at ([yshift=.2 cm]v2j) {$v_{2_j}$};
\node [draw=black, shape=circle, fill=white] (v2n) at (-3.4,-5.5) {};
\node [scale=.7] at ([xshift=-.3 cm]v2n) {$v_{2_n}$};
\node at (-5.8,-5.5) {$\ldots$};
\node at (-4.4,-5.5) {$\ldots$};
\node [scale=.7] at (-5,-5) {$G_1$};

\node at (-2.5,-5.5) {$\ldots$};
\node at (2.5,-5.5) {$\ldots$};
\node [scale=.7] at (0,-6.8) {$S(G,2)=(V_2,E_2)$};

\draw (0,-5.5) ellipse (7cm and 1.8cm);

\end{tikzpicture}
\caption{Sketch of a polymeric Sierpi\'{n}ski graph $P(G,2)$, where a small ellipse represents a copy of $G$ and the dashed lines connecting a vertex $a_{i_k}$ and a small ellipse mean that each vertex of the $k$-th  copy of $G$ in $S(G,2)$ is connected to $a_{i_k}$.}\label{FigP(G,2)}
\end{figure}

%%%%%%%%%%
The authors of  \cite{GeneralizedSierpinski} announced some results about generalized Sierpi\'{n}ski graphs concerning their automorphism groups  and perfect codes. In our opinion, these results definitely deserve to be published. Later, the total chromatic number of generalized Sierpi\'{n}ski graphs was  studied  in \cite{Geetha2015} and the strong metric dimension has recently been  studied in \cite{Rodriguez-Velazquez2016}.
The authors of \cite{Rodriguez-Velazquez2015a} obtained closed formulae for the chromatic, vertex cover, clique and domination numbers of generalized Sierpi\'{n}ski graphs $S(G,t)$ in terms of parameters of the base graph $G$. More recently,   a general upper bound on the Roman domination number of $S(G,t)$ was obtained in \cite{Ramezani2016}. In particular, it was studied the case in which the base graph $G$ is a path, a cycle, a complete graph or a graph having exactly one universal vertex. To the best of our knowledge, \cite{Rodriguez-Velazquez2015} is the first
published paper studying the generalized Sierpi\'{n}ski graphs. In that article, the authors obtained closed formulae for the Randi\'{c} index $R_{-1/2}$ of $S(G,t)$ and $P(G,t)$, where the base graph $G$ is a complete graph, a triangle-free regular graph or a bipartite semiregular graph. 
The present article is a continuation of \cite{Rodriguez-Velazquez2015}
where we study the general Randi\'{c} index $R_{\alpha}$ of  Sierpi\'nski-type polymeric networks.
 In particular,  we obtain closed  formulae for the general Randi\'{c} index $R_{\alpha}$ of  $S(G,t)$ and $P(G,t)$, where the base graph $G$, parameter $t$ and exponent  $\alpha$ are arbitrary.

\section{Computing the General Randi\'c index of $S(G,t)$ }

%%%%%%%%%%%%  Lemma for general graph
From now on, given a graph $G=(V,E)$ and a specific edge $\{x,y\}\in E$, the number of copies of $\{x,y\}$ in $S(G,t)$, where vertex $x$ has degree $d(x)+l$ and vertex $y$ has degree $d(y)+l'$ will be denoted by $f_{S(G,t)}\left(d(x)+l,d(y)+l'\right)$, where $l,l'\in \{0,1\}$. For instance, for the graph $S(G,3)$ shown in Figure \ref{FigS(G,3).} we have $f_{S(G,3)}\left(d(1),d(2)\right)=21$ and $f_{S(G,3)}\left(d(1)+1,d(2)\right)=f_{S(G,3)}\left(d(1),d(2)+1\right)=f_{S(G,3)}\left(d(1)+1,d(2)+1\right)=12$. It can be noted that copies of a specific edge $\{x,y\}\in E$ in $S(G,t)$ are edges $\{wxy^r, wyx^r\}$, where $r\in\{0,1,\ldots,t-1\}$ and $w\in V^{t-1-r}$.

The set of neighbours that $x\in V$ has in $G$ will be denoted by $N(x)$, \textit{i.e.}, $$N(x)=\{z\in V:\; \{x,z\}\in E\}.$$

Given two vertices $u,v\in V $, the number of triangles of $G$ containing $u$ and $v$ will be denoted by $\tau(u,v)$ and the number of triangles of $G$ will be denoted by $\tau(G)$.\footnote{A triangle in a graph is a set of three vertices whose induced subgraph is isomorphic to $K_3$.} Note that $\displaystyle\sum_{\{u,v\}\in E(G)}\tau(u,v)=3\tau(G)$.
The complexity of counting the number of triangles of a graph $G$ is polynomial with respect to $|V|$. The trivial approach of counting the number of triangles is to check for every triple $x,y,z\in V$ if $x,y,z$ forms a triangle. This procedure gives us the algorithmic complexity of $O(n^3)$. However, this algorithmic complexity can be improved \cite{Chiba1985}. 

Notice that for any pair of adjacent vertices $u,v\in V$ we have $| N(u)\cap N(v)|=\tau(u,v)$, $| N(u)\cup N(v)|=d(u)+d(v)-\tau(u,v)$ and $| N(u)- N(v)|=d(u)-\tau(u,v)$. Given a graph of order $n$, from now on we will use the function $\psi(t)=1+n+n^2+\cdots+n^{t-1}=\dfrac{n^t-1}{n-1}$. 

\begin{lemma}\label{SierpinskiGeneral}
For any integer $t\ge 2$ and any edge  $\{x,y\}$ of   a graph $G$ of order $n$,   
\begin{enumerate}[{\rm (i)}]
\item $f_{S(G,t)}\left(d(x),d(y)\right)=n^{t-2}\left(n-d(x)-d(y)+\tau(x,y)\right).$
\item $f_{S(G,t)}\left(d(x),d(y)+1\right)=n^{t-2}\left(d(y)-\tau(x,y)\right)-\psi(t-2)d(x).$
\item $f_{S(G,t)}\left(d(x)+1,d(y)\right)=n^{t-2}\left(d(x)-\tau(x,y)\right)-\psi(t-2)d(y).$
\item $f_{S(G,t)}\left(d(x)+1,d(y)+1\right)=n^{t-2}\left(\tau(x,y)+1\right)+\psi(t-2)\left(d(x)+d(y)+1\right).$
\end{enumerate}
\end{lemma}

\begin{proof}
There are two different possibilities for the degree of any vertex   $zx$ of $S(G,2)$, namely $d(x)$ and $d(x)+1$, \textit{i.e.},   $zx$ has degree $d(x)+1$ for all $z\in N(x)$ and  $zx$ has degree $d(x)$ for all $z\not\in N(x)$.  Therefore,

\begin{description}
\item $f_{S(G,2)}\left(d(x),d(y)\right)=\left|V-\left(N(x)\cup N(y)\right)\right|=n- d(x)-d(y)+\tau(x,y),$
\item $f_{S(G,2)}\left(d(x),d(y)+1\right)=\left|N(y)-N(x)\right|=d(y)-\tau(x,y)$,
\item $f_{S(G,2)}\left(d(x)+1,d(y)\right)=\left|N(x)-N(y)\right|=d(x)-\tau(x,y)$,
\item $f_{S(G,2)}\left(d(x)+1,d(y)+1\right)=|N(x)\cap N(y)|+1=\tau(x,y)+1$
\end{description}

For any $t\ge 3$,  $w\in V^{t-1}$ and $z\in V$,  the degree of $zw$ in $S(G,t)$ coincides with the degree of $w$ in $S(G,t-1)$, except when $w=jj\cdots j$ is an extreme vertex and $z  \in N(j)$, in which case the degree of $zw=zjj\cdots j$ in $S(G,t)$ is $d(j)+1$  while   the degree of $w=jj\cdots j$ in $S(G,t-1)$ is $d(j)$. 
Hence, we deduce the following:
\begin{flalign*}
f_{S(G,t)}\left(d(x),d(y)\right)&=nf_{S(G,t-1)}\left(d(x),d(y)\right)&\\
&=n^{t-2}\left(n- d(x)-d(y)+\tau(x,y)\right),
\end{flalign*}
\begin{flalign*}
f_{S(G,t)}\left(d(x),d(y)+1\right)&=nf_{S(G,t-1)}\left(d(x),d(y)+1\right)-d(x)&\\
&=n^{t-2}\left(d(y)-\tau(x,y)\right)-\sum_{i=0}^{t-3}n^id(x)\\
&=n^{t-2}\left(d(y)-\tau(x,y)\right)-\psi(t-2)d(x),
\end{flalign*}
\begin{flalign*}
f_{S(G,t)}\left(d(x)+1,d(y)\right)&=nf_{S(G,t-1)}\left(d(x)+1,d(y)\right)-d(y)&\\
&=n^{t-2}\left(d(x)-\tau(x,y)\right)-\sum_{i=0}^{t-3}n^id(y)\\
&=n^{t-2}\left(d(x)-\tau(x,y)\right)-\psi(t-2)d(y),
\end{flalign*}
and finally,
\begin{flalign*}
f_{S(G,t)}\left(d(x)+1,d(y)+1\right)&=nf_{S(G,t-1)}\left(d(x)+1,d(y)+1\right)+d(x)+d(y)+1&\\
&=n^{t-2}\left(\tau(x,y)+1\right)+\sum_{i=0}^{t-3}n^i\left(d(x)+d(y)+1\right)\\
&=n^{t-2}\left(\tau(x,y)+1\right)+\psi(t-2)\left(d(x)+d(y)+1\right).
\end{flalign*}
Therefore, the result follows.
\end{proof}

From the previous result we can deduce that the number of copies of $\{x,y\}\in E$ in $S(G,t)$ is $\displaystyle\sum_{l,l'\in\{0,1\}}f_{S(G,t)}\left(d(x)+l,d(y)+l'\right)=\psi(t)$. This result could also have been obtained by counting of these edges, considering the form of the words of vertices associated with them.

The main result of this section is Theorem \ref{MainTheorem-S(G,t)} which provides a formula for the general Randi\'c index of $S(G,t)$, where the base graph $G$ and $\alpha$ are arbitrary, and $t$ is an integer greater than one. 

\begin{theorem}\label{MainTheorem-S(G,t)}
For any graph $G$ of order $n\ge 2$ and any integer $t\ge 2$,
$$R_\alpha(S(G,t))=\sum_{\{x,y\}\in E}W_{\{x,y\}},$$
where
\begin{flalign*}
W_{\{x,y\}}&=n^{t-2}\left(n- d(x)-d(y)+\tau(x,y)\right)d(x)^\alpha d(y)^\alpha+&\\
&+\left(n^{t-2}\left(d(y)-\tau(x,y)\right)-\psi(t-2)d(x)\right)d(x)^\alpha(d(y)+1)^\alpha+\\
&+\left(n^{t-2}\left(d(x)-\tau(x,y)\right)-\psi(t-2)d(y)\right)(d(x)+1)^\alpha d(y)^\alpha+\\
&+\left(n^{t-2}\left(\tau(x,y)+1\right)+\psi(t-2)\left(d(x)+d(y)+1\right)\right)(d(x)+1)^\alpha(d(y)+1)^\alpha.
\end{flalign*}
\end{theorem}
\begin{proof}
Since any copy of a vertex $z\in V$ in $S(G,t)$ has degree $d(z)$ or $d(z)+1$, the general Randi\'c index of  $S(G,t)$ can be expressed as
$$R_\alpha(S(G,t))=\sum_{\{x,y\}\in E}\sum_{i=0}^1\sum_{j=0}^1 f_{S(G,t)}\left(d(x)+i,d(y)+j\right)(d(x)+i)^\alpha (d(y)+j)^\alpha.$$
Hence, by Lemma \ref{SierpinskiGeneral} the result immediately follows.
\end{proof}

It can be noted that closed formulas can be obtained for several topological indices of $S(G,t)$ similarly to the previous proof. For instance, we are referring to the atom-bond connectivity index (ABC), geometric arithmetic index (GA), Harmonic index (H), first Zagreb index ($M_1$) and sum-connectivity index ($\chi$).

The remaining results of this section are directly derived  from Theorem \ref{MainTheorem-S(G,t)}.

%%%%%%  Theorem Regular graph
\begin{corollary}\label{ThSierpinskiRegular}
For any $\delta$-regular graph $G$ of order $n$ and any integer $t\ge 2,$
\begin{align*}
R_\alpha(S(G,t))&=\left(\frac{n^{t-1}\delta}{2}(n-2\delta)+3n^{t-2}\tau(G)\right)\delta^{2\alpha}\\
&+\left(\left(n^{t-1}+\psi(t-1)\right) \delta^{2}-6n^{t-2}\tau(G)\right)\delta^\alpha(\delta+1)^{\alpha} \\
&+\left(\dfrac{n\delta}{2}\psi(t-1)+n\delta^{2}\psi(t-2)+3n^{t-2}\tau(G)\right)(\delta+1)^{2\alpha}.
\end{align*}
\end{corollary}

A complete graph $K_n$ of order $n\ge 2$ is $(n-1)$-regular and it has $\displaystyle {n \choose 3}$ triangles.  Therefore, the next result follows.

\begin{remark}
For any integers $t,n\ge 2$, $$R_\alpha(S(K_n,t))=n^{\alpha+1}(n-1)^{\alpha+1}+\frac{n^{2\alpha+t+1}-2n^{2(\alpha+1)}+n^{2\alpha+1}}{2}$$
\end{remark}

\begin{remark}
For any integers $t\ge 2$ and $n\ge 4$, 
$$R_\alpha(S(C_n,t))=4^\alpha n^{t-1}(n-4)+4\cdot 6^\alpha\left(n^{t-1}-n\psi(t-2)\right)+9^\alpha n\left(\psi(t-1)+4\psi(t-2)\right).$$
\end{remark}

%%%%%%%%%%%%%%%%  Semiregular bipartite graph
\begin{corollary}\label{ThSierpinskiBipartite}
Let $G=(U_1\cup U_2,E)$ be a bipartite $(\delta_1,\delta_2)$-semiregular   graph of order $n=n_1+n_2$, where $|U_1|=n_1$ and $|U_2|=n_2$. Then for any integer $t\ge 2$,
  
\begin{align*}
R_\alpha(S(G,t))&=n_1 n^{t-2}\delta_1^{\alpha+1}\delta_2^\alpha (n-\delta_1-\delta_2)+n_1\delta_1^{\alpha+1}(\delta_2+1)^\alpha\left(\delta_2n^{t-2}-\delta_1\psi(t-2)\right)\\
&+n_2(\delta_1+1)^\alpha\delta_2^{\alpha+1}\left(\delta_1 n^{t-2}-\delta_2\psi(t-2)\right)\\
&+n_1\delta_1(\delta_1+1)^\alpha(\delta_2+1)^\alpha\left(n^{t-2}+(\delta_1+\delta_2+1)\psi(t-2)\right). 
\end{align*}
\end{corollary}
Chemical trees are trees that have no vertex with degree greater than 4. The graph $S(K_{1,3},2)$ is an example of a chemical tree. Notice that for any $t\ge 2$, the Sierpi\'{n}ski graph $S(K_{1,3},t)$, is a chemical tree.

As a particular case of Corollary \ref{ThSierpinskiBipartite} we obtain the following result. 
\begin{remark} For any integers $r,t\ge 2$,
\begin{align*}
R_\alpha(S(K_{1,r},t))&=(r+1)^\alpha\left((r+1)^{t-1}(r-1)+1\right)+2^\alpha r^{\alpha+1}+\\
&+\left(2(r+1)\right)^\alpha\left(2(r+1)^{t-1}-r-2\right).
\end{align*}
\end{remark}

\begin{corollary}
Let $t,n$ be integers such that $t\ge 2$ and $n\ge 3$. Then
$$R_\alpha(S(P_2,t))=2^{\alpha+1}+(2^t-3)2^{2\alpha}$$
and
\begin{align*}
R_\alpha(S(P_n,t))&=2^\alpha n^{t-2}(n-3)\left(2^\alpha n-2^{\alpha+2}+2\right)+\\
&+3^\alpha\left[2^{\alpha+2}(n-3)\left(n^{t-2}-\psi(t-2)\right)+4n^{t-2}-2\psi(t-2)\right]+\\
&+2^{2\alpha+1}\left(n^{t-2}-2\psi(t-2)\right)+\\
&+3^\alpha\left[3^\alpha(n-3)\left(n^{t-2}+5\psi(t-2)\right)+2^{\alpha+1}\left(n^{t-2}+4\psi(t-2)\right)\right].
\end{align*}
\end{corollary}

In order to continue presenting our results, we need to introduce a definition. Given a graph $G$ on the vertex set $V$, we define the parameter 
$$M_\alpha(G)=\sum_{x\in V}d(x)^\alpha.$$ 
Note that $M_1(G)$ is equal to twice the number of edges of $G$ and $M_2(G)$ is the first Zagreb index. Considering that for any graph $G$ of maximum degree $\Delta(G)$ and minimum degree $\delta(G)$, and any vertex $x\in V(G)$ we have that $d(x)^\alpha+(\delta(G)+1)^\alpha-\Delta(G)^\alpha\le (d(x)+1)^\alpha\le d(x)^\alpha+(\Delta(G)+1)^\alpha-\delta(G)^\alpha$ and replacing in Theorem \ref{MainTheorem-S(G,t)}, we deduce the following bounds. 

\begin{theorem}
For any triangle free graph $G$ of order $n\ge 2$, maximum degree $\Delta$ and minimum degree $\delta$, and any integer $t\ge 2$,
$$\beta_L(S(G,t))\le R_\alpha(S(G,t))\le \beta_U(S(G,t)),$$
where
\begin{align*}
\beta_L(S(G,t))&=n^{t-2}(n-\Delta)R_\alpha(G)+\\
&+2\left(n^{t-2}\delta-\Delta\psi(t-2)\right)\left(R_\alpha(G)+M_{\alpha+1}(G)\left[(\delta+1)^\alpha-\Delta^\alpha\right]\right)+\\
&+\left(n^{t-2}+(2\delta+1)\psi(t-2)\right)\left(R_\alpha(G)+2M_{\alpha+1}(G)\left[(\delta+1)^\alpha-\Delta^\alpha\right]\right)+\\
&+\left(n^{t-2}+(2\delta+1)\psi(t-2)\right)\dfrac{M_1(G)}{2}\left[(\delta+1)^\alpha-\Delta^\alpha\right]^2,
\end{align*}
and
\begin{align*}
\beta_U(S(G,t))&=n^{t-2}(n-\delta)R_\alpha(G)+\\
&+2\left(n^{t-2}\Delta-\delta\psi(t-2)\right)\left(R_\alpha(G)+M_{\alpha+1}(G)\left[(\Delta+1)^\alpha-\delta^\alpha\right]\right)+\\
&+\left(n^{t-2}+(2\Delta+1)\psi(t-2)\right)\left(R_\alpha(G)+2M_{\alpha+1}(G)\left[(\Delta+1)^\alpha-\delta^\alpha\right]\right)+\\
&+\left(n^{t-2}+(2\Delta+1)\psi(t-2)\right)\dfrac{M_1(G)}{2}\left[(\Delta+1)^\alpha-\delta^\alpha\right]^2.
\end{align*}
Moreover, $\beta_L(S(G,t))=R_\alpha(S(G,t))=\beta_U(S(G,t))$ if and only if $G$ is a triangle free regular graph.
\end{theorem}

\section{Computing the General Randi\'c index of $P(G,t)$}

The main result of this section is Theorem \ref{MainTheorem-P(G,t)} which provides a formula for the general Randi\'c index of $P(G,t)$, where the base graph $G$ and $\alpha$ are arbitrary, and $t$ is an integer greater than one. 

From now on, given a graph $G$ and a vertex $x\in V(G)$, we refer to the degree of $x$ in $G$ as $d_G(x)$. If there is no  ambiguity, we will continue writing $d(x)$.

\begin{remark}
For any graph $G$ of order $n\ge 2$,
$$R_\alpha(P(G,1))=n^\alpha\sum_{x\in V}(d(x)+1)^\alpha + \sum_{\{x,y\}\in E}(d(x)+1)^\alpha(d(y)+1)^\alpha.$$
\end{remark}

\begin{proof}
Note that the vertex $a_{1_1}$ has degree $n$ and it is adjacent to all vertices of $G$. Hence, for every vertex $x\in V$, there is one edge $\{a_{1_1},x\}$, where $a_{1_1}$ has degree $n$ and $x$ has degree $d_{P(G,1)}(x)=d_G(x)+1$. For the remaining edges $\{x,y\}$ we have that $d_{P(G,1)}(x)=d_G(x)+1$ and $d_{P(G,1)}(y)=d_G(y)+1$. Therefore, the result follows. 
\end{proof}

\begin{corollary}\label{coroPRegularG1}
For any $\delta$-regular graph $G$ of order $n\ge 2$,
$$R_\alpha(P(G,1))=n^{\alpha+1}(\delta+1)^\alpha+\frac{n\delta(\delta+1)^{2\alpha}}{2}.$$
In particular, for complete graphs, 
$$R_\alpha(P(K_n,1))=\frac{n^{2\alpha+1}(n+1)}{2}.$$
\end{corollary}

\begin{corollary}\label{CoroSemiRegularG1}
Let $G=(U_1\cup U_2,E)$ be a bipartite $(\delta_1,\delta_2)$-semiregular graph of order $n=n_1+n_2$, where $|U_1|=n_1$ and $|U_2|=n_2$.  Then,
$$R_\alpha(P(G,1))=n^\alpha\left(n_1(\delta_1+1)^\alpha + n_2(\delta_2+1)^\alpha\right) + n_1\delta_1\left((\delta_1+1)(\delta_2+1)\right)^\alpha.$$
\end{corollary}

Given a graph $G$ and a vertex $x\in V$, the number of copies of  $x$ having degree  $r$ in $S(G,t)$ will be denoted by $g_{_{S(G,t)}}(r)$.

\begin{lemma}\label{SierpinskiVertexDegree}
For any graph $G$ of order $n$, any vertex $x\in V$ and any integer $t\ge 2,$
\begin{enumerate}[{\rm (i)}]
\item $g_{_{S(G,t)}}(d(x))=n^{t-1}-d(x)\psi(t-1).$
\item $g_{_{S(G,t)}}(d(x)+1)=d(x)\psi(t-1).$
\end{enumerate}
\end{lemma}

\begin{proof}
There are two different possibilities for the degree of any copy of vertex $x$ in $S(G,t)$, namely $d(x)$ and $d(x)+1$. Every copy of vertex $x$ having degree $d(x)+1$ in $S(G,t)$ is a vertex of the form $z_1\cdots z_kyx\cdots x$, where $y\in N(x)$, $0\le k\le t-2$ and $z_i\in V(G)$ for $1\le i\le k$. So,
$$g_{_{S(G,t)}}(d(x)+1)=d(x)\sum_{i=0}^{t-2}n^i=d(x)\psi(t-1).$$
On the other hand, every copy of vertex $x$ in $S(G,t)$ is a vertex of the form $z_1\cdots z_{t-1}x$, where $z_i\in V(G)$ for $1\le i\le t-1$. Hence, we have $n^{t-1}$ copies of vertex $x$ in $S(G,t)$, and as a consequence,
$$g_{_{S(G,t)}}(d(x))=n^{t-1}-g_{_{S(G,t)}}(d(x)+1)=n^{t-1}-d(x)\psi(t-1).$$
\end{proof}

Now, we state a formula for the general Randi\'c of $P(G,t)$, where the base graph $G$ and the exponent $\alpha$ are arbitrary, and $t$ is an integer greater than one.

\begin{theorem}\label{MainTheorem-P(G,t)}
For any graph $G=(V,E)$ of order $n\ge 2$ and any integer $t\ge 2$,
$$R_\alpha(P(G,t))=\sum\limits_{i=1}^{7}\beta_{i},$$
where
\begin{flalign*}
\beta_1&=n^\alpha\sum_{x\in V}(d(x)+2)^\alpha,\quad \beta_2=\sum_{\{x,y\}\in E}(d(x)+2)^\alpha(d(y)+2)^\alpha,&
\end{flalign*}
\begin{flalign*}
\beta_3&=(n+1)^\alpha n\psi(t-2)\sum_{x\in V}(d(x)+2)^\alpha+&\\
&+\dfrac{(n+1)^\alpha\left(t-2-n\psi(t-2)\right)}{(n-1)}\sum_{x\in V}d(x)(d(x)+2)^\alpha+\\
&+\dfrac{(n+1)^\alpha\left(t-2-n\psi(t-2)\right)}{(1-n)}\sum_{x\in V}d(x)(d(x)+3)^\alpha,
\end{flalign*}
$\displaystyle\beta_4=\sum_{\{x,y\}\in E}W_{4\{x,y\}}$ with 
\begin{flalign*}
W_{4\{x,y\}}&=(d(x)+2)^\alpha(d(y)+2)^\alpha\left(n-d(x)-d(y)+\tau(x,y)\right)\psi(t-2)+&\\
&+(d(x)+2)^\alpha(d(y)+3)^\alpha\left(\left(d(y)-\tau(x,y)\right)\psi(t-2)+d(x)\dfrac{t-2-\psi(t-2)}{n-1}\right)+\\
&+(d(x)+3)^\alpha(d(y)+2)^\alpha\left(\left(d(x)-\tau(x,y)\right)\psi(t-2)+ d(y)\dfrac{t-2-\psi(t-2)}{n-1}\right)+\\
&+(d(x)+3)^\alpha(d(y)+3)^\alpha\left(\tau(x,y)+1\right)\psi(t-2)+\\
&+(d(x)+3)^\alpha(d(y)+3)^\alpha\left(d(x)+d(y)+1\right)\dfrac{t-2-\psi(t-2)}{1-n},
\end{flalign*}
\begin{flalign*}
\beta_5&=(n+1)^\alpha\psi(t-1)\sum_{x\in V}(d(x)+2)^\alpha+&\\
&+\dfrac{(n+1)^\alpha(t-1-\psi(t-1))}{(n-1)}\sum_{x\in V}d(x)(d(x)+2)^\alpha+\\
&+\dfrac{(n+1)^\alpha(t-1-\psi(t-1))}{(1-n)}\sum_{x\in V}d(x)(d(x)+3)^\alpha,
\end{flalign*}
\begin{flalign*}
\beta_6&=(n+1)^\alpha\sum_{x\in V}(d(x)+1)^\alpha\left(n^{t-1}-d(x)\psi(t-1)\right)+&\\
&+(n+1)^\alpha\psi(t-1)\sum_{x\in V}d(x)(d(x)+2)^\alpha,
\end{flalign*} and
$\displaystyle\beta_7=\sum_{\{x,y\}\in E}W_{7\{x,y\}}$ with 
\begin{flalign*}
W_{7\{x,y\}}&=(d(x)+1)^\alpha(d(y)+1)^\alpha n^{t-2}\left(n-d(x)-d(y)+\tau(x,y)\right)+&\\
&+(d(x)+1)^\alpha(d(y)+2)^\alpha\left(n^{t-2}\left(d(y)-\tau(x,y)\right)-\psi(t-2)d(x)\right)+\\
&+(d(x)+2)^\alpha(d(y)+1)^\alpha\left(n^{t-2}\left(d(x)-\tau(x,y)\right)-\psi(t-2)d(y)\right)+\\
&+(d(x)+2)^\alpha(d(y)+2)^\alpha\left(n^{t-2}\left(\tau(x,y)+1\right)+\psi(t-2)\left(d(x)+d(y)+1\right)\right).
\end{flalign*}
\end{theorem}

\begin{proof}
Let $d(u),d(v)$ be degrees of $u,v$ in $P(G,t)$, respectively. We differentiate the following cases for any edge $\{u,v\}$ of $P(G,t)$. 
 
\begin{enumerate}
\item $u=a_{1_1}$ and $v\in V_1.$ In this case, there are $n$ edges $\{u,v\}$ with $d(u)=n$ and $d_{P(G,t)}(v)=d_G(v)+2.$ Then the contribution of these edges to the General Randi\'{c} index is equal to $\beta_1$.

\item $u,v \in V_1.$ In this case, each extreme vertex of an edge $\{u,v\}$ has degree $d_{P(G,t)}(u)=d_G(u)+2$ and $d_{P(G,t)}(v)=d_G(v)+2$. So, the contribution of these edges to the General Randi\'{c} index is equal to $\beta_2$.

\item $u\in A_i$ and $v\in V_i$ for $2\le i\le t-1$. We assume that $v=wx$, where $w\in V^{i-1}$ and $x\in V$. In this case $d(u)=n+1$ and, by Lemma \ref{SierpinskiVertexDegree}, there are $g_{_{S(G,i)}}(d(x)+1)=d(x)\psi(i-1)$ edges $\{u,v\}$ where $v$ has degree $d(v)=d(x)+3$ and there are $g_{_{S(G,i)}}(d(x))=n^{i-1}-d(x)\psi(i-1)$ edges $\{u,v\}$ where $d(v)=d(x)+2$. Thus, the contribution of these edges to the General Randi\'{c} index is equal to $\displaystyle\sum_{i=2}^{t-1}\sum_{x\in V}W'_{3\{x\}}$, where, 
$$W'_{3\{x\}}=(n+1)^\alpha\sum_{i=0}^1(d(x)+i+2)^\alpha g_{_{S(G,i)}}(d(x)+i).$$
Since $\displaystyle\sum_{i=2}^{t-1}n^{i-1}=n\psi(t-2)$ and $\displaystyle\sum_{i=2}^{t-1}\psi(i-1)=\dfrac{t-2-n\psi(t-2)}{1-n}$, we obtain $\displaystyle\beta_3=\sum_{i=2}^{t-1}\sum_{x\in V}W'_{3\{x\}}$.

\item $u,v\in V_i,$ for $2\leq i\leq t-1$. We assume that $u=wx$ and $v=w'y$, where $w,w'\in V^{i-1}$ and $\{x,y\}\in E$.
By Lemma \ref{SierpinskiGeneral}, there are
\begin{description}
\item $f_{S(G,i)}\left(d(x),d(y)\right)=n^{i-2}\left(n-d(x)-d(y)+\tau(x,y)\right)$ edges $\{u,v\}$ where $d(u)=d(x)+2$ and $d(v)=d(y)+2$,
\item $f_{S(G,i)}\left(d(x),d(y)+1\right)=n^{i-2}\left(d(y)-\tau(x,y)\right)-\psi(i-2)d(x)$ edges $\{u,v\}$ where $d(u)=d(x)+2$ and $d(v)=d(y)+3$,
\item $f_{S(G,i)}\left(d(x)+1,d(y)\right)=n^{i-2}\left(d(x)-\tau(x,y)\right)-\psi(i-2)d(y)$ edges $\{u,v\}$ where $d(u)=d(x)+3$ and $d(v)=d(y)+2$,
\item $f_{S(G,i)}\left(d(x)+1,d(y)+1\right)=n^{i-2}\left(\tau(x,y)+1\right)+\psi(i-2)\left(d(x)+d(y)+1\right)$ edges $\{u,v\}$ where $d(u)=d(x)+3$ and $d(v)=d(y)+3$.
\end{description}
Hence, the contribution of these edges to the General Randi\'{c} index is equal to $\displaystyle\sum_{i=2}^{t-1}\sum_{\{x,y\}\in E}W'_{4\{x,y\}}$, where,
$$W'_{4\{x,y\}}=\sum_{i=0}^1\sum_{j=0}^1 (d(x)+i+2)^\alpha(d(y)+j+2)^\alpha f_{S(G,i)}(d(x)+i,d(y)+j).$$
Since $\displaystyle\sum_{i=2}^{t-1}n^{i-2}=\psi(t-2)$ and $\displaystyle\sum_{i=2}^{t-1}\psi(i-2)=\dfrac{t-2-\psi(t-2)}{1-n}$, we obtain $\displaystyle\beta_4=\sum_{i=2}^{t-1}\sum_{\{x,y\}\in E}W'_{4\{x,y\}}=\sum_{\{x,y\}\in E}W_{4\{x,y\}}$.

\item $u\in A_{i+1}$ and $v\in V_i$ for $1\leq i \leq t-1$. We assume that $v=wx$, where $w\in V^{i-1}$ and $x\in V$.
In this case $d(u)=n+1$ and, by Lemma \ref{SierpinskiVertexDegree}, there are $g_{_{S(G,i)}}(d(x)+1)=d(x)\psi(i-1)$ edges $\{u,v\}$ where $v$ has degree $d(v)=d(x)+3$ and there are $g_{_{S(G,i)}}(d(x))=n^{i-1}-d(x)\psi(i-1)$ edges $\{u,v\}$ where $d(v)=d(x)+2$. Hence, the contribution of these edges to the General Randi\'{c} index is equal to $\displaystyle\sum_{i=2}^{t-1}\sum_{x\in V}W'_{5\{x\}}$, where, 
$$W'_{5\{x\}}=(n+1)^\alpha\sum_{i=0}^1(d(x)+i+2)^\alpha g_{_{S(G,i)}}(d(x)+i).$$
Since $\displaystyle\sum_{i=1}^{t-1}n^{i-1}=\psi(t-1)$ and $\displaystyle\sum_{i=1}^{t-1}\psi(i-1)=\dfrac{t-1-n\psi(t-1)}{1-n}$, we obtain $\displaystyle\beta_5=\sum_{i=2}^{t-1}\sum_{x\in V}W'_{5\{x\}}$.  

\item $u\in A_t$ and $v\in V_t$. We assume that $v=wx$, where $w\in V^{t-1}$ and $x\in V$. As above $d(u)=n+1$ and, by Lemma \ref{SierpinskiVertexDegree}, there are $g_{_{S(G,t)}}(d(x)+1)=d(x)\psi(t-1)$ edges $\{u,v\}$ where $v$ has degree $d(v)=d(x)+2$ and there are $g_{_{S(G,t)}}(d(x))=n^{t-1}-d(x)\psi(t-1)$ edges $\{u,v\}$ where $d(v)=d(x)+1$. Thus, the contribution of these edges to the General Randi\'{c} index is equal to $\beta_6$.

\item $u,v\in V_t$. We assume that $u=wx$ and $v=w'y$, where $w,w'\in V^{t-1}$ and $x,y\in V$.
By Lemma \ref{SierpinskiGeneral}, there are
\begin{description}
\item $f_{S(G,t)}\left(d(x),d(y)\right)=n^{t-2}\left(n-d(x)-d(y)+\tau(x,y)\right)$ edges $\{u,v\}$ where $d(u)=d(x)+1$ and $d(v)=d(y)+1$,
\item $f_{S(G,t)}\left(d(x),d(y)+1\right)=n^{t-2}\left(d(y)-\tau(x,y)\right)-\psi(t-2)d(x)$ edges $\{u,v\}$ where $d(u)=d(x)+1$ and $d(v)=d(y)+2$,
\item $f_{S(G,t)}\left(d(x)+1,d(y)\right)=n^{t-2}\left(d(x)-\tau(x,y)\right)-\psi(t-2)d(y)$ edges $\{u,v\}$ where $d(u)=d(x)+2$ and $d(v)=d(y)+1$,
\item $f_{S(G,t)}\left(d(x)+1,d(y)+1\right)=n^{t-2}\left(\tau(x,y)+1\right)+\psi(t-2)\left(d(x)+d(y)+1\right)$ edges $\{u,v\}$ where $d(u)=d(x)+2$ and $d(v)=d(y)+2$.
\end{description}
Hence, the contribution of these edges to the General Randi\'{c} index is equal to $\displaystyle\beta_7=\sum_{\{x,y\}\in E}W_{7\{x,y\}}$.

\end{enumerate}
According to the seven cases above, the result follows. 
\end{proof}

Now, we will show some particular cases of Theorem \ref{MainTheorem-P(G,t)}.

\begin{corollary}\label{ThePolymerRegular}
For any $\delta$-regular graph $G$ of order $n\geq2$ and any integer $t\ge 2$,
$$R_\alpha(P(G,t))=\sum\limits_{i=1}^{7}\beta_{i},$$
where
\begin{flalign*}
\beta_{1}&=n^{\alpha+1}(\delta+2)^\alpha,\quad \beta_{2}=\dfrac{n\delta(\delta+2)^{2\alpha}}{2},&
\end{flalign*}
\begin{flalign*}
\beta_{3}&=n^2\psi(t-2)(n+1)^\alpha(\delta+2)^\alpha+&\\
&+\frac{n\delta(n+1)^\alpha(\delta+2)^\alpha\left(t-2-n\psi(t-2)\right)}{(n-1)}+\\
&+\dfrac{n\delta(n+1)^\alpha(\delta+3)^\alpha(t-2-n\psi(t-2))}{(1-n)},
\end{flalign*}
\begin{flalign*}
\beta_4&=(\delta+2)^{2\alpha}\psi(t-2)\left(\frac{n\delta(n-2\delta)}{2}+3\tau(G)\right)+&\\
&+(\delta+2)^\alpha(\delta+3)^\alpha\left(\left(n\delta^2-6\tau(G)\right)\psi(t-2)+\dfrac{n\delta^2(t-2-\psi(t-2))}{n-1}\right)+\\
&+(\delta+3)^{2\alpha}\left(\left(3\tau(G)+\frac{n\delta}{2}\right)\psi(t-2)+\dfrac{n\delta(2\delta+1)(t-2-\psi(t-2))}{2(1-n)}\right),
\end{flalign*}
\begin{flalign*}
\beta_{5}&=\frac{n(n+1)^\alpha(\delta+2)^\alpha\left(\delta(t-1)+\psi(t-1)(n-\delta-1)\right)}{(n-1)}+&\\
&+\frac{n\delta(n+1)^\alpha(\delta+3)^\alpha\left(t-1-\psi(t-1)\right)}{(1-n)},
\end{flalign*}
\begin{flalign*}
\beta_{6}&=(n+1)^\alpha(\delta+1)^\alpha\left(n^{t}-n\delta\psi(t-1)\right)+n\delta(n+1)^\alpha(\delta+2)^\alpha\psi(t-1),&
\end{flalign*} and
\begin{flalign*}
\beta_7&=(\delta+1)^{2\alpha}n^{t-2}\left(\frac{n\delta}{2}(n-2\delta)+3\tau(G)\right)+&\\
&+(\delta+1)^\alpha(\delta+2)^\alpha\left(\delta^2\left(n^{t-1}-n\psi(t-2)\right)-6n^{t-2}\tau(G)\right)+\\
&+(\delta+2)^{2\alpha}\left(\frac{n\delta}{2}\psi(t-1)+n\delta^{2}\psi(t-2)+3n^{t-2}\tau(G)\right).
\end{flalign*}
\end{corollary}

As we mentioned above, a complete graph $K_n$ of order $n\ge 2$ is $(n-1)$-regular and it has $\displaystyle {n \choose 3}$ triangles. Therefore, the next result is deduced from Corollary \ref{ThePolymerRegular}.

\begin{corollary}
For any integers $n,t\ge 2$,
$$R_\alpha(P(K_n,t))=\displaystyle\sum_{l=1}^7\beta_l,$$ where 
\begin{flalign*}
\beta_1&=n^{\alpha+1}(n+1)^\alpha,\quad \beta_2=\frac{n(n-1)(n+1)^{2\alpha}}{2},&
\end{flalign*}
\begin{flalign*}
\beta_3&=n(t-2)(n+1)^{2\alpha}+n(n+1)^\alpha(n+2)^\alpha(2+n\psi(t-2)-t),&
\end{flalign*}
\begin{flalign*}
\beta_4&=(t-2)n(n-1)(n+1)^\alpha(n+2)^\alpha+\frac{(n+2)^{2\alpha}}{2}\left(n^3\psi(t-2)+(t-2)(n-2n^2)\right),&
\end{flalign*}
\begin{flalign*}
\beta_5&=(t-1)n(n+1)^{2\alpha}+n(n+1)^\alpha(n+2)^\alpha\left(\psi(t-1)-(t-1)\right),&
\end{flalign*}
\begin{flalign*}
\beta_6&=n^{\alpha+1}(n+1)^\alpha+(n^t-n)(n+1)^{2\alpha},&
\end{flalign*}
\begin{flalign*}
\beta_7&=(n-1)n^{\alpha+1}(n+1)^\alpha+\frac{(n^{t+1}-2n^2+n)(n+1)^{2\alpha}}{2}.&
\end{flalign*}
\end{corollary}


\begin{thebibliography}{10}
\expandafter\ifx\csname url\endcsname\relax
  \def\url#1{\texttt{#1}}\fi
\expandafter\ifx\csname urlprefix\endcsname\relax\def\urlprefix{URL }\fi

\bibitem{Blumen}
A.~Blumen, A.~Jurjiu, Multifractal spectra and the relaxation of model polymer
  networks, The Journal of Chemical Physics 116~(6) (2002) 2636--2641. \href{http://dx.doi.org/10.1063/1.1433744}{doi: 10.1063/1.1433744}


\bibitem{camarda}
K.~V. Camarda, C.~D. Maranas, Optimization in polymer design using connectivity
  indices, Industrial \& Engineering Chemistry Research 38~(5) (1999)
  1884--1892.\href{http://dx.doi.org/10.1021/ie980682n}{doi:10.1021/ie980682n}

\bibitem{Chiba1985}
N. Chiba, T. Nishizeki, Arboricity and subgraph listing algorithms, SIAM Journal on Computing 14~(1) (1985)  210--223. \href{http://dx.doi.org/10.1137/0214017}{doi:10.1137/0214017}

\bibitem{Das2004}
K.~C. Das, I.~Gutman, Some properties of the second {Z}agreb index, MATCH Communications in Mathematical and in Computer Chemistry~(52) (2004) 103--112.
\newline\urlprefix\url{http://match.pmf.kg.ac.rs/electronic_versions/Match52/match52_103-112.pdf}

\bibitem{Rodriguez-Velazquez2016}
E.~Estaji, J.~A. Rodr\'{\i}guez-Vel\'{a}zquez, {The strong metric dimension of generalized Sierpi\'{n}ski graphs with pendant vertices}, Ars Mathematica Contemporanea 12~(1) (2017) 127--134.
\newline\urlprefix\url{http://amc-journal.eu/index.php/amc/article/view/813}

\bibitem{gao-lu}
J.~Gao, M.~Lu, On the Randi\'c index of unicyclic graphs, MATCH Communications in Mathematical and in Computer Chemistry 53~(2) (2005) 377--384.
\newline\urlprefix\url{http://match.pmf.kg.ac.rs/electronic_versions/Match53/n2/match53n2_377-384.pdf}

\bibitem{Geetha2015}
J.~Geetha, K.~Somasundaram, Total coloring of generalized {S}ierpi\'nski graphs, The Australasian Journal of Combinatorics 63 (2015) 58--69.
\newline\urlprefix\url{https://ajc.maths.uq.edu.au/pdf/63/ajc_v63_p058.pdf}

\bibitem{Gravier...Parreau}
S.~Gravier, M.~Kov\v{s}e, M.~Mollard, J.~Moncel, A.~Parreau, New results on variants of covering codes in Sierpi\'{n}ski graphs, Designs, Codes and Cryptography 69~(2) (2013) 181--188. \href{http://dx.doi.org/10.1007/s10623-012-9642-1}{doi:10.1007/s10623-012-9642-1}


\bibitem{GeneralizedSierpinski}
S.~Gravier, M.~Kov\v{s}e, A.~Parreau, Generalized Sierpi\'{n}ski graphs, EuroComb 2011 (Poster), R\'enyi Institute, Budapest, 2011.
\newline\urlprefix\url{http://www.renyi.hu/conferences/ec11/posters/parreau.pdf}

\bibitem{Gutman2008}
I.~Gutman, B.~Furtula (eds.), Recent Results in the Theory of Randi\'c Index in MATHEMATICAL CHEMISTRY MONOGRAPHS, No.~6, University of Kragujevac and Faculty of Science Kragujevac, 2008.
\newline\urlprefix\url{http://match.pmf.kg.ac.rs/mcm6.htm}


\bibitem{Gutman2001}
I.~Gutman, M.~Lepovi\'c, Choosing the exponent in the definition of the connectivity index, Journal of the Serbian Chemical Society 66~(9) (2001) 605--611.
\newline\urlprefix\url{http://www.shd.org.rs/JSCS/Vol66/No9/V66-No9-05.pdf}

\bibitem{Klavzar2016(survey)}
A.M. Hinz, S. Klav\v{z}ar, S.S. Zemlji\v{c}, A survey and classification of Sierpi\'{n}ski-type graphs, Discrete Applied Mathematics 217 (2017) 565--600. \href{http://dx.doi.org/10.1016/j.dam.2016.09.024}{doi:10.1016/j.dam.2016.09.024}


\bibitem{Hinz-Parisse}
A.~Hinz, D.~Parisse, The Average Eccentricity of Sierpi\'{n}ski Graphs, Graphs and Combinatorics 28~(5) (2012) 671--686. \href{http://dx.doi.org/10.1007/s00373-011-1076-4}{doi:10.1007/s00373-011-1076-4}


\bibitem{Hirtz-Holz}
A.~M. Hinz, C.~H. auf~der Heide, An efficient algorithm to determine all shortest paths in Sierpi\'{n}ski graphs, Discrete Applied Mathematics 177 (2014) 111--120. \href{http://dx.doi.org/10.1016/j.dam.2014.05.049}{doi:10.1016/j.dam.2014.05.049}


\bibitem{hu-li}
Y.~Hu, X.~Li, Y.~Shi, T.~Xu, I.~Gutman, On molecular graphs with smallest and greatest zeroth-order general Randi\'c index, MATCH Communications in Mathematical and in Computer Chemistry 54~(2) (2055) 425--434.
\newline\urlprefix\url{http://match.pmf.kg.ac.rs/electronic_versions/Match54/n2/match54n2_425-434.pdf}


\bibitem{JU}
A.~Jurjiu, T.~Koslowski, A.~Blumen, Dynamics of deterministic fractal polymer networks: Hydrodynamic interactions and the absence of scaling, The Journal of Chemical Physics 118~(5) (2003) 2398--2404. \href{http://dx.doi.org/10.1063/1.1534576}{doi:10.1063/1.1534576}


\bibitem{JU1}
A.~G. Jurjiu, Dynamics of polymer networks modelled by finite regular fractals, Ph.D. thesis, Fakult\"{a}t f\"{u}r Mathematik und Physik, Albert-Ludwigs-Universit\"{a}t Freiburg im Breisgau (2005).
\newline\urlprefix\url{https://www.freidok.uni-freiburg.de/fedora/objects/freidok:2734/datastreams/FILE1/content}

\bibitem{Klavzar1997}
S.~Klav\v{z}ar, U.~Milutinovi\'{c}, Graphs $S(n,k)$ and a variant of the Tower of Hanoi problem, Czechoslovak Mathematical Journal 47~(1) (1997) 95--104. 
\newline\urlprefix\url{http://dml.cz/dmlcz/127341}


\bibitem{Klavzar2002}
S.~Klav\v{z}ar, U.~Milutinovi\'c, C.~Petr, 1-perfect codes in Sierpi\'{n}ski graphs, Bulletin of the Australian Mathematical Society 66~(3) (2002) 369--384. \href{http://dx.doi.org/10.1017/S0004972700040235}{doi:10.1017/S0004972700040235}


\bibitem{Klavsar-Peterin}
S.~Klav\v{z}ar, I.~Peterin, S.~S. Zemlji\v{c}, Hamming dimension of a graph--the case of Sierpi\'{n}ski graphs, European Journal of Combinatorics 34~(2) (2013) 460--473. \href{http://dx.doi.org/10.1016/j.ejc.2012.09.006}{doi:10.1016/j.ejc.2012.09.006}


\bibitem{Klavsar-Zeljic}
S.~Klav\v{z}ar, S.~S. Zemlji\v{c}, On distances in Sierpi\'{n}ski graphs: Almost-extreme vertices and metric dimension, Applicable Analysis and Discrete Mathematics 7~(1) (2013) 72--82. \href{http://dx.doi.org/10.2298/AADM130109001K}{doi:10.2298/AADM130109001K}


\bibitem{Li2005}
H.~Li, M.~Lu, The $m$-connectivity index of graphs, MATCH Communications in Mathematical and in Computer Chemistry 54~(2) (2005) 417--423.
\newline\urlprefix\url{http://match.pmf.kg.ac.rs/electronic_versions/Match54/n2/match54n2_417-423.pdf}


\bibitem{Li2006}
X.~Li, I.~Gutman, Mathematical Aspects of Randi{\'c}-type Molecular Structure Descriptors, Mathematical Chemistry Monographs, No. 1, University of Kragujevac and Faculty of Science Kragujevac, 2006.
\newline\urlprefix\url{http://match.pmf.kg.ac.rs/mcm1.htm}


\bibitem{Li2008}
X.~Li, Y.~Shi, A survey on the Randi\'c index, MATCH Communications in Mathematical and in Computer Chemistry 59~(1) (2008) 127--156.
\newline\urlprefix\url{http://match.pmf.kg.ac.rs/electronic_versions/Match59/n1/match59n1_127-156.pdf}


\bibitem{liu-gutman}
B.~Liu, I.~Gutman, On general Randi\'c indices, MATCH Communications in Mathematical and in Computer Chemistry 58~(1) (2007) 147--154.
\newline\urlprefix\url{http://match.pmf.kg.ac.rs/electronic_versions/Match58/n1/match58n1_147-154.pdf}

\bibitem{J.Liu}
J.~Liu, The Randi\'c index and girth of triangle-free graphs, Ars Combinatoria 113 (2014) 289--297.
\newline\urlprefix\url{http://www.combinatorialmath.ca/arscombinatoria/vol113.html}

\bibitem{Parisse2009}
D. Parisse,  On some metric properties of the Sierpi\'{n}ski graphs $S(n, k)$, Ars Combinatoria 90 (2009) 145--160.
\newline\urlprefix\url{http://www.combinatorialmath.ca/arscombinatoria/vol90.html}

\bibitem{Ramezani2016}
F.~Ramezani, E.~D. Rodriguez-Bazan, J.~A. Rodr\'{i}guez-Vel\'{a}zquez, On the {R}oman domination number of generalized {S}ierpi\'{n}ski graphs, Filomat, to appear.
\newline\urlprefix\url{https://arxiv.org/pdf/1605.06918.pdf}

\bibitem{Randic1975}
M.~Randi\'c, On characterization of molecular branching, Journal of the American Chemical Society 97~(23) (1975) 6609--6615. \href{http://dx.doi.org/10.1021/ja00856a001}{doi:10.1021/ja00856a001}


\bibitem{Randic2008}
M.~Randi\'c, On history of the Randi\'c index and emerging hostility toward chemical graph theory, MATCH Communications in Mathematical and in Computer Chemistry 59~(1) (2008) 5--124.
\newline\urlprefix\url{http://match.pmf.kg.ac.rs/content59n1.htm} 


\bibitem{Rodriguez2005a}
J.~A. Rodr\'{\i}guez, A spectral approach to the Randi\'c index, Linear Algebra and its Applications 400 (2005) 339--344. \href{http://dx.doi.org/10.1016/j.laa.2005.01.003}{doi:10.1016/j.laa.2005.01.003}

\bibitem{Rodriguez2005b}
J.~A. Rodr\'{\i}guez, J.~M. Sigarreta, On the Randi\'c index and conditional parameters of a graph, MATCH Communications in Mathematical and in Computer Chemistry 54~(2) (2005) 403--416.
\newline\urlprefix\url{http://match.pmf.kg.ac.rs/electronic_versions/Match54/n2/match54n2_403-416.pdf}

\bibitem{Rodriguez-Velazquez2015a}
J.~A. Rodr\'{i}guez-Vel\'{a}zquez, E.~D. Rodr\'{i}guez-Bazan,
  A.~Estrada-Moreno, {On Generalized Sierpi\'{n}ski Graphs}, Discussiones Mathematicae Graph Theory 37~(3) (2017) 547--560. \href{https://doi.org/10.7151/dmgt.1945}{doi:10.7151/dmgt.1945}

\bibitem{Rodriguez-Velazquez2015}
J.~A. Rodr\'iguez-Vel\'azquez, J.~Tom\'as-Andreu, On the Randi\'c index of polymer networks modelled by generalized Sierpi\'nski graphs, MATCH Communications in Mathematical and in Computer Chemistry 74~(1) (2015) 145--160.
\newline\urlprefix\url{http://match.pmf.kg.ac.rs/electronic_versions/Match74/n1/match74n1_145-160.pdf}


\bibitem{LibroPolymers}
R.~Stepto (ed.), Polymer Networks. Principles of their Formation, Structure and Properties, Springer Netherlands, 1998.
\newline\urlprefix\url{http://www.springer.com/in/book/9780751402643}


\bibitem{Todeschini2008}
R.~Todeschini, V.~Consonni, Handbook of Molecular Descriptors, Methods and Principles in Medicinal Chemistry, WILEY-VCH Verlag GmbH, 2008. \href{http://dx.doi.org/10.1002/9783527613106}{doi:10.1002/9783527613106}


\bibitem{Wang}
W.~Wang, X.~Hou, W.~Ning, The $k$-connectivity index of an infinite class of dendrimer nanostars, Digest Journal of Nanomaterials and Biostructures 6~(3) (2011) 1199--1205.
\newline\urlprefix\url{http://www.chalcogen.ro/1199_Wang.pdf}


\bibitem{Yero2010b}
I.~G. Yero, J.~A. Rodr\'iguez-Vel\'azquez, I.~Gutman, Estimating the higher-order Randi\'c index, Chemical Physics Letters 489~(1--3) (2010) 118--120. \href{http://dx.doi.org/10.1016/j.cplett.2010.02.052}{doi:10.1016/j.cplett.2010.02.052}


\end{thebibliography}
\end{document}